\documentclass[11pt]{article}

\usepackage{amsmath, amsthm}
\usepackage{amsmath, amsfonts}
\usepackage{amsmath, amssymb}
\usepackage{amsmath}
\usepackage{graphics}
\usepackage{graphicx}
\usepackage{color}
\usepackage{hyperref}
\usepackage[all]{xy}

\newtheorem{theorem}{Theorem}[subsection]

\newtheorem{proto-definition}[theorem]{Proto-Definition}
\newtheorem{pseudo-definition}[theorem]{Pseudo-Definition}
\newtheorem{definition-lemma}[theorem]{Definition/Lemma}
\newtheorem{definition-explanation}[theorem]{Definition/Explanation}
\newtheorem{explanation-definition}[theorem]{Explanation/Definition}
\newtheorem{definition-fact}[theorem]{Definition/Fact}
\newtheorem{definition-notation}[theorem]{Definition/Notation}
\newtheorem{definition-conjecture}[theorem]{Definition/Conjecture}
\newtheorem{definition-theorem}[theorem]{Definition/Theorem}

\newtheorem{lemma-definition}[theorem]{Lemma/Definition}

\newtheorem{remark-notation}[theorem]{\it Remark/Notation}

\newtheorem{application-lemma}[theorem]{Application/Lemma}

\newtheorem{example-definition}[theorem]{Example/Definition}

\newtheorem{definition-prototype}[theorem]{Definition-Prototype}

\numberwithin{equation}{subsection}

\newtheorem{stheorem}{Theorem}[section]
\newtheorem{sdefinition}[stheorem]{Definition}
\newtheorem{sproto-definition}[stheorem]{Proto-Definition}
\newtheorem{spseudo-definition}[stheorem]{Pseudo-Definition}
\newtheorem{sdefinition-lemma}[stheorem]{Definition/Lemma}
\newtheorem{sdefinition-explanation}[stheorem]{Definition/Explanation}
\newtheorem{sexplanation-definition}[stheorem]{Explanation/Definition}
\newtheorem{sdefinition-fact}[stheorem]{Definition/Fact}
\newtheorem{sdefinition-notation}[stheorem]{Definition/Notation}
\newtheorem{sdefinition-conjecture}[stheorem]{Definition/Conjecture}
\newtheorem{sdefinition-theorem}[stheorem]{Definition/Theorem}
\newtheorem{slemma}[stheorem]{Lemma}
\newtheorem{slemma-definition}[stheorem]{Lemma/Definition}

\newtheorem{scorollary}[stheorem]{Corollary}

\newtheorem{sremark-notation}[stheorem]{\it Remark/Notation}
\newtheorem{sconjecture}[stheorem]{Conjecture}

\newtheorem{sapplication-lemma}[stheorem]{Application/Lemma}

\newtheorem{sexample}[stheorem]{Example}
\newtheorem{sexample-definition}[stheorem]{Example/Definition}

\newtheorem{sdefinition-prototype}[stheorem]{Definition-Prototype}

\newtheorem{ssproto-definition}[sstheorem]{Proto-Definition}
\newtheorem{sspseudo-definition}[sstheorem]{Pseudo-Definition}
\newtheorem{ssdefinition-lemma}[sstheorem]{Definition/Lemma}
\newtheorem{ssdefinition-explanation}[sstheorem]{Definition/Explanation}
\newtheorem{ssexplanation-definition}[sstheorem]{Explanation/Definition}
\newtheorem{ssdefinition-fact}[sstheorem]{Definition/Fact}
\newtheorem{ssdefinition-notation}[sstheorem]{Definition/Notation}
\newtheorem{ssdefinition-conjecture}[sstheorem]{Definition/Conjecture}
\newtheorem{ssdefinition-theorem}[sstheorem]{Definition/Theorem}

\newtheorem{sslemma-definition}[sstheorem]{Lemma/Definition}

\newtheorem{ssremark-notation}[sstheorem]{\it Remark/Notation}

\newtheorem{ssapplication-lemma}[sstheorem]{Application/Lemma}

\newtheorem{ssexample-definition}[sstheorem]{Example/Definition}

\newtheorem{ssdefinition-prototype}[sstheorem]{Definition-Prototype}

\newcommand{\End}{\mbox{\it End}\,}

\newcommand{\Hom}{\mbox{\it Hom}\,}

\newcommand{\Id}{\mbox{\it Id}\,}

\newcommand{\Image}{\mbox{\it Im}\,}

\newcommand{\Spec}{\mbox{\it Spec}\,}

\newcommand{\determinant}{\mbox{\it det}\,}
\newcommand{\dimm}{\mbox{\it dim}\,}

\newcommand{\pr}{\mbox{\it pr}}

\begin{document}

\enlargethispage{24cm}

\begin{titlepage}

$ $

\vspace{-1.5cm} 

\noindent\hspace{-1cm}
\parbox{6cm}{\small August 2015}\
   \hspace{7cm}\
   \parbox[t]{6cm}{yymm.nnnnn [math.DG] \\
                D(11.3.1): smooth map  \\
				}

\vspace{2cm}

\centerline{\large\bf
 Further studies on}
\vspace{1ex}
\centerline{\large\bf
 the notion of differentiable maps from Azumaya/matrix manifolds}
\vspace{1ex}
\centerline{\large\bf
 I.\ The smooth case}

\vspace{3em}

\centerline{\large
  Chien-Hao Liu
            \hspace{1ex} and \hspace{1ex}
  Shing-Tung Yau
}

\vspace{4em}

\begin{quotation}
\centerline{\bf Abstract}

\vspace{0.3cm}

\baselineskip 12pt  
{\small
 In this follow-up of our earlier two works
    D(11.1) (arXiv:1406.0929 [math.DG])  and
    D(11.2) (arXiv:1412.0771 [hep-th])
      in the D-project,
 we study further
   the notion of a `differentiable map from an Azumaya/matrix manifold to a real manifold'.
 A conjecture is made
  that  the notion of differentiable maps from Azumaya/matrix manifolds
    as defined in D(11.1)
    is equivalent to one defined through the contravariant ring-homomorphisms alone.
 A proof  of this conjecture for the smooth (i.e.\ $C^{\infty}$) case is given in this note.
 Thus, at least in the smooth case,
  our setting for D-branes in the realm of differential geometry
  is completely parallel to that in the realm of algebraic geometry,
  cf.\ arXiv:0709.1515 [math.AG] and arXiv:0809.2121 [math.AG].
 A related conjecture on such maps to ${\Bbb R}^n$, as a $C^k$-manifold,
  and its proof in the $C^{\infty}$ case is also given.
 As a by-product, a conjecture on a division lemma in the finitely differentiable case
  that generalizes the division lemma in the smooth case from Malgrange
  is given in the end, as well as other comments on the conjectures in the general $C^k$ case.
 We remark that there are similar conjectures in general and theorems in the smooth case
   for the fermionic/super generalization of the notion.
} 
\end{quotation}

\vspace{9.6em}

\baselineskip 12pt
{\footnotesize
} 

 \bigskip

\noindent {\small MSC number 2010: 58A40, 14A22, 81T30; 51K10, 16S50, 46L87.
} 

\bigskip

\baselineskip 10pt
{\scriptsize
\noindent{\bf Acknowledgements.}
We thank Cumrun Vafa and Andrew Strominger for lectures
  that influence our understanding on themes in string theory and gravity.
C.-H.L.\ thanks in addition
  Yuan-Pin Lee, Hui-Wen Lin, Chin-Lung Wang
     for discussions on the algebro-geometric aspect of the notion
     and the algebraic Gromov-Witten type theory from D-strings;	
  Wu-Yen Chuang
     for discussions on stability conditions; 	
  Yng-Ing Lee, Mao-Pei Tsui
     for discussions on fuzzy Lagrangian submanifolds, harmonic D-branes and issues beyond;
  Pei-Ming Ho
     for discussions on technical issues in noncommutative geometry and literature guide;
  Chieh-Hsiung Kuan, Shr-Wei Liu, Ai-Nung Wang for other discussions/conversations;
  Heng-Yu Chen, Kerwin Hui, Chen-Te Ma, Chien-Hsun Wang, Chong-Tang Wu
     for pinning him down to details, challenging his thought,
	 drawing his attention to several themes and literatures that may be related;   	
 the Department of Mathematics and the Department of Physics at the National Taiwan University
     for hospitality, May 2015;
 Ling-Miao Chou
     for comments on illustrations and moral support.
 The project is supported by NSF grants DMS-9803347 and DMS-0074329.
} 

\end{titlepage}

\newpage

\begin{titlepage}

$ $

\vspace{12em}

\centerline{\includegraphics[width=0.24\textwidth]{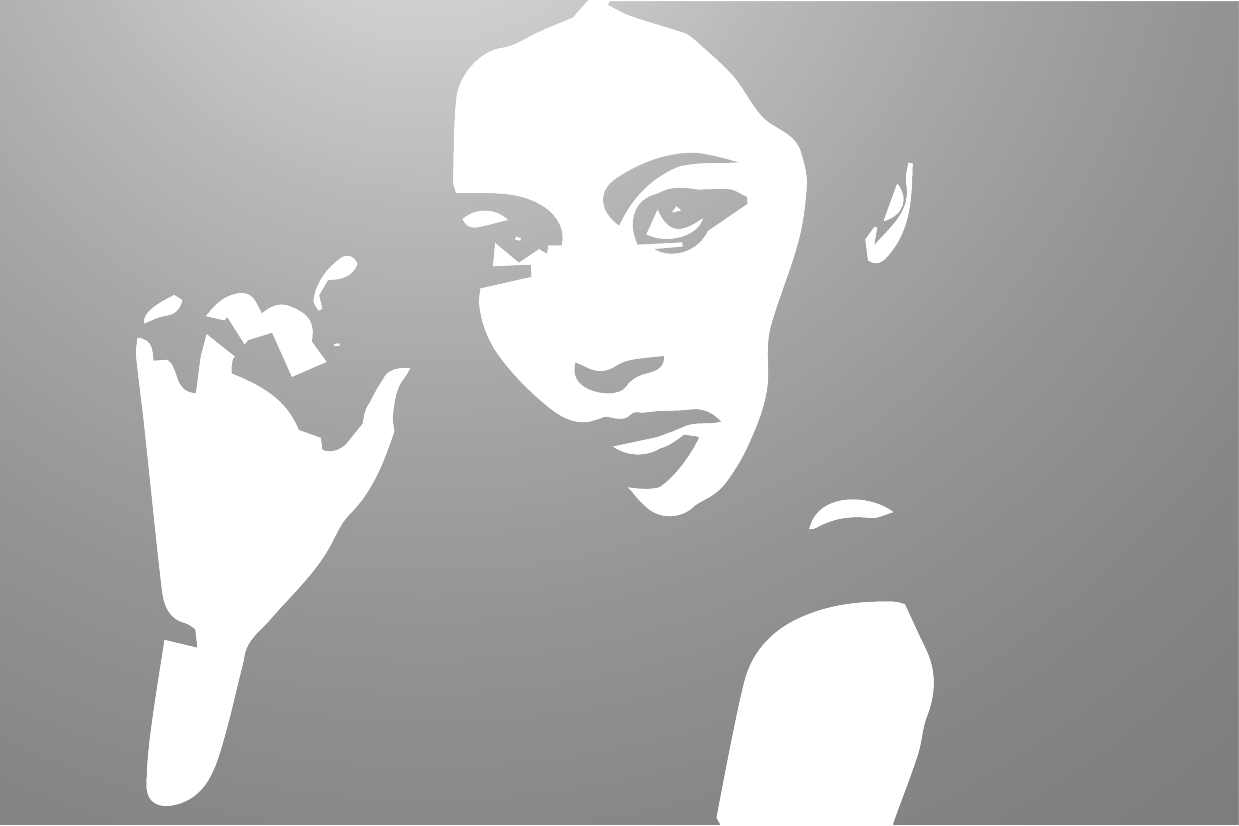}}

\bigskip

\centerline{\small\it
 Chien-Hao Liu dedicates this note to the memory of}
\centerline{\small\it
 Beina Yao $\;(1981$--$2015)$.}

\vspace{24em}

\baselineskip 11pt

{\footnotesize
\noindent
(From C.H.L.)
 Playing the flute, as well as other woodwind instruments,
   is really meant to sing through the piece of metal (or wood).
 This is why when Beina's singing came across my world by accident in late 2014,
   it immediately attracted and touched me.
 Her voice has a rich and solid timbre
    with the flexibility to cover a wide range
    usually only a mechanical musical instrument can reach without losing its elegance, cleanness, and fullness.
 Despite years of rigorous conservatory training in Chinese traditional folk music,
  her voice remains so natural and her singing looks so effortless.
 (An over-trained voice tends to become unhuman, which I never liked.)
 But she won't be that unique if this is all she got.
 What's truly inspiring and respectable is
  her brave twist of her path in music to follow her own heart and avoid a safe but otherwise mediocre life,
  her courage to battle with her cancer,  and
  her noble and philanthropic legacy in the end of her life,
    donating what her body could offer to enable other people to see the world again.
 At this moment of another return-to-zero/origin in both the project and my life,
  her example rings particularly deep in my heart.
 Last but not the least,
  my respect and consolation to her parents, who reminded me of a musical family that influenced me forever.
 The major part of the current note was typed with the company of her songs.
 } 

\end{titlepage}


\newpage
$ $

\vspace{-3em}

\centerline{\sc
 Notion of Differentiable Maps from Azumaya/Matrix Manifolds I: Smooth Case
 } %

\vspace{2em}


\begin{flushleft}
{\Large\bf 0. Introduction and outline}
\end{flushleft}
The notion of differentiable maps from an Azumaya/matrix manifold
 (with a fundamental module) to a real manifold
 was developed in [L-Y2] (D(11.1))
 as a natural mathematical language to describe D-branes as a dynamical/fundamental object in string theory.
(See [Liu] for a review.)
Its fermionic/super generalization was given  in [L-Y3] (D(11.2)).
At the first sight and in comparison with the setting in [L-Y1] (D(1)),
 the mathematical setting for D-branes in the realm of differential geometry
  look more involved/constrained.
The core notion behind is the notion of
 `{\it $C^k$-admissible ring-homomorphisms}' from the $C^k$-function ring of a $C^k$-manifold
 to  the Azumaya/matrix-function ring of an Azumaya/matrix $C^k$-manifold;
 ([L-Y2: Definition~5.1.2] (D(11.1), reviewed in Definition~1.2, Sec.\ 1 of the current note).
The setup of this notion meant to deal with the technical issue
     that a noncommutative ring cannot be made a $C^k$-ring
 and was guided by the principle
    that any good notion of a `map' should be accompanied by a corresponding natural notion
     of the `graph' of the map.
	
 In this follow-up of [L-Y2] (D(11.1)) and [L-Y3] (D(11.2)), 	
 we examine further
   the notion of a `differentiable map from an Azumaya/matrix manifold to a real manifold'.
 A conjecture is made
  that  the notion of differentiable maps from Azumaya/matrix manifolds
    as defined in [L-Y2] (D(11.1))
    is equivalent to one defined through the contravariant ring-homomorphisms alone.
 A proof  of this conjecture for the smooth (i.e.\ $C^{\infty}$) case is given in this note.
 Thus, at least in the smooth case,
  our setting for D-branes in the realm of differential geometry
  is completely parallel to that in the realm of algebraic geometry,
  cf.\ [L-Y1] (D(1)) and [L-L-S-Y] (D(2)).
 A related conjecture on such maps to ${\Bbb R}^n$, as a $C^k$-manifold,
  and its proof in the $C^{\infty}$ case is also given.
 As a by-product, a conjecture on a division lemma in the finitely differentiable case
  that generalizes the division lemma in the smooth case from Malgrange
  is given in the end, as well as other comments on the conjectures in the general $C^k$ case.

 We remark that there are similar conjectures in general and theorems in the smooth case
   for the fermionic/super generalization of the notion.

\bigskip

\bigskip

\noindent
{\bf Convention.}
 References for standard notations, terminology, operations and facts in
    (1) algebraic geometry;
    (2) synthetic geometry, $C^{\infty}$-algebraic geometry;
    (3) D-branes
 can be found respectively in
    (1) [Ha];
    (2) [Du1], [Du2], [Jo], [Ko], [M-R];
    (3) [Po].
 \begin{itemize}
  \item[$\cdot$]
   For clarity, the {\it real line} as a real $1$-dimensional manifold is denoted by ${\Bbb R}^1$,
    while the {\it field of real numbers} is denoted by ${\Bbb R}$.
   Similarly, the {\it complex line} as a complex $1$-dimensional manifold is denoted by ${\Bbb C}^1$,
    while the {\it field of complex numbers} is denoted by ${\Bbb C}$.
	
  \item[$\cdot$]	
  The inclusion `${\Bbb R}\hookrightarrow{\Bbb C}$' is referred to the {\it field extension
   of ${\Bbb R}$ to ${\Bbb C}$} by adding $\sqrt{-1}$, unless otherwise noted.

  \item[$\cdot$]
   The {\it complexification} of an ${\Bbb R}$-module $M$ is denoted by
    $M^{\Bbb C}\;(:= M\otimes_{\Bbb R}{\Bbb C})$.

 \item[$\cdot$]	
  The {\it real $n$-dimensional vector spaces} ${\Bbb R}^{\oplus n}$
      vs.\ the {\it real $n$-manifold} $\,{\Bbb R}^n$; \\
  similarly, the {\it complex $r$-dimensional vector space ${\Bbb C}^{\oplus r}$}
     vs.\ the {\it complex $r$-fold} $\,{\Bbb C}^r$.

 \item[$\cdot$]
  All $C^k$-manifolds, $k\in{\Bbb Z}_{\ge 0}\cup\{\infty\}$,
     are paracompact, Hausdorff, admitting a (locally finite) partition of unity,
     and embeddable into some ${\Bbb R}^N$ as closed $C^k$-submanifolds.
  We adopt the {\it index convention for tensors} from differential geometry.
   In particular, the tuple coordinate functions on an $n$-manifold is denoted by, for example,
   $(y^1,\,\cdots\,y^n)$.
  However, no up-low index summation convention is used.

  \item[$\cdot$]
   `{\it differentiable}' = $k$-times differentiable (i.e.\ $C^k$)
         for some $k\in{\Bbb Z}_{\ge 1}\cup\{\infty\}$;
   `{\it smooth}' $=C^{\infty}$;
   $C^0$ = {\it continuous} by standard convention.

  \item[$\cdot$]
   $\Spec R $ ($:=\{\mbox{prime ideals of $R$}\}$)
         of a commutative Noetherian ring $R$  in algebraic geometry\\
   vs.\ $\Spec R$ of a $C^k$-ring $R$
  ($:=\Spec^{\Bbb R}R :=\{\mbox{$C^k$-ring homomorphisms $R\rightarrow {\Bbb R}$}\}$).

  \item[$\cdot$]
  {\it morphism} between schemes in algebraic geometry
    vs.\ {\it $C^k$-map} between $C^k$-manifolds or $C^k$-schemes
         	in differential topology and geometry or $C^k$-algebraic geometry.
			
  \item[$\cdot$]			
   {\it matrix} $m$ vs.\ manifold of {\it dimension} $m$.
\end{itemize}

\bigskip

\begin{flushleft}
{\bf Outline}
\end{flushleft}
\nopagebreak
{\small
 \baselineskip 12pt  
 \begin{itemize}
    \item[1]
     Conjectures on the notion of $C^k$-maps from Azumaya/matrix $C^k$-manifolds	 	
   	
    \item[2]
	 Preliminaries: When $X$ is a point	
	  \vspace{-.6ex}
	  \begin{itemize}
	   \item[\Large$\cdot\;$]
         The canonical $C^k$-ring structure on a finite-dimensional ${\Bbb R}$-algebra

	   \item[\Large$\cdot\;$]	
         Validity of Conjecture~1.3 when $X$ is a point
						
       \item[\Large$\cdot\;$]
         Validity of Conjecture~1.5 when $X$ is a point
	  \end{itemize}
	
    \item[3]	
	 Proof of Conjectures in the $C^{\infty}$ case
	  \vspace{-.6ex}
	  \begin{itemize}
        \item[3.1]		
		 Proof of Conjecture~1.3 in the $C^{\infty}$ case

	    \item[3.2]
	     Proof of Conjecture~1.5 in the $C^{\infty}$ case		
	  \end{itemize}
	    	
	\item[4]
     Remarks on the general $C^k$ case	
	  \vspace{-.6ex}
	  \begin{itemize}
	    \item[\Large$\cdot\;$]
          Reflections on $C^{\infty}$- vs.\ general $C^k$-algebraic geometry, and the proof

	    \item[\Large$\cdot\;$]	
          A conjecture on a division lemma in the finitely differentiable case
	  \end{itemize}
 \end{itemize}
} 

\newpage

\section{Conjectures on the notion of $C^k$-maps from Azumaya/matrix $C^k$-manifolds}

Let
 \begin{itemize}
  \item[\LARGE $\cdot$]
   $X$, $Y$ be $C^k$-manifolds, $k\in {\Bbb Z}_{\ge 0}\cup\{\infty\}$,
    with their $C^k$-function ring denoted $C^k(X)$ and $C^k(Y)$ respectively;
   (and their structure sheaf ${\cal O}_X$ and ${\cal O}_Y$ respectively); 	
	
  \item[\LARGE $\cdot$]	
   $E$ be a complex $C^k$ vector bundle on $X$ of rank $r$; \\ and
   $C^k(\End_{\Bbb C}(E))$ be its endomorphism algebra.
 \end{itemize}

\bigskip

Note that
 if an ${\Bbb R}$-subalgebra $S\subset C^k(\End_{\Bbb C}(E))$
    admits a $C^k$-ring structure,
 that structure is unique up to $C^k$-ring isomorphisms.

\bigskip

\begin{sdefinition} {\bf [weakly $C^k$-admissible ring-homomorphism].} \rm
 A ring-homomorphism
  $$
   \varphi^{\sharp}\; :\; C^k(Y)\longrightarrow\;  C^k(\End_{\Bbb C}(E))
  $$
  over ${\Bbb R}\hookrightarrow{\Bbb C}$
  is said to be {\it weakly $C^k$-admissible}
 if the ${\Bbb R}$-subalgebra $\Image\varphi^{\sharp}$ of $C^k(\End_{\Bbb C}(E))$
  admits a $C^k$-ring structure
  with respect to which $\varphi^{\sharp}$ is a $C^k$-ring-homomorphism.
\end{sdefinition}

\medskip

\begin{sdefinition} {\bf [$C^k$-admissible ring-homomorphism].} \rm\\
 (Cf.\ [L-Y2: Definition~5.1.2] (D(11.1)).)$\;$
 A ring-homomorphism
    $$
	  \xymatrix{
	   C^k(\End_{\Bbb C}(E))   &&& C^k(Y)\ar[lll]_-{\varphi^{\sharp}}
	   }
	$$
    over ${\Bbb R}\hookrightarrow {\Bbb C}$
    is said to be {\it $C^k$-admissible}	
    if it extends to the following commutative diagram of ring-homomorphisms
	 $$
		 \xymatrix{
	       C^k(\End_{\Bbb C}(E))
			     &&& C^k(Y) \ar[lll]_-{\varphi^{\sharp}}
			                                      \ar@{_{(}->}^-{pr_Y^{\sharp}}[d]   \\			    
			   \rule{0ex}{1em}C^k(X) \ar@{^{(}->}[u]
			                                                                 \ar@{^{(}->}[rrr]_-{pr_X^{\sharp}}
				 &&& C^k(X\times Y) \ar[lllu]_-{\tilde{\varphi}^{\sharp}}		
		}
	 $$
	 such that
	  \begin{itemize}
	   \item[(1)]
	      $\;\tilde{\varphi}^{\sharp}: C^k(X\times Y)
	                  \rightarrow \Image(\tilde{\varphi}^{\sharp})$
			    is  a $C^k$-normal quotient
		 (i.e.\ the $C^k$-ring structure on $C^k(X\times Y)$ descends
		            to a $C^k$-ring structure on $\Image(\tilde{\varphi}^{\sharp})$ ),
 			
      \item[(2)]
	     replacing $C^k(\End_{\Bbb C}(E))$ with
             $$
			  A_{\varphi}\;
			     =\; C^k(X)\langle \Image(\varphi^{\sharp})\rangle\;
			    :=\; \Image(\tilde{\varphi}^{\sharp})\;\;\;
			   \subset\; C^k(\End_{\Bbb C}(E))
			   $$
			  with the $C^k$-ring structure induced from that of $C^k(X\times Y)$ by Condition (1),
          then 			
  	    $$
		 \xymatrix{
	       A_{\varphi}
			     &&& C^k(Y) \ar[lll]_-{\varphi^{\sharp}}
			                                      \ar@{_{(}->}^-{pr_Y^{\sharp}}[d]   \\			    
			   \rule{0ex}{1em}C^k(X) \ar@{^{(}->}[u]
			                                                                 \ar@{^{(}->}[rrr]_-{pr_X^{\sharp}}
				 &&& C^k(X\times Y) \ar@{->>}[lllu]_-{\tilde{\varphi}^{\sharp}}		
		}
	   $$
	   is a commutative diagram of $C^k$-ring homomorphisms.
   \end{itemize}	
\end{sdefinition}

\bigskip

Clearly, for a correspondence $\varphi^{\sharp}:C^k(Y)\rightarrow C^k(\End_{\Bbb C}(E))$,
 $$
  \begin{array}{cl}
      & \mbox{$\varphi^{\sharp}$ is a $C^k$-admissible ring-homomorphism
  	                      over ${\Bbb R}\hookrightarrow {\Bbb C}$} \\[.6ex]
	\Longrightarrow
      & \mbox{$\varphi^{\sharp}$ is a weakly $C^k$-admissible ring-homomorphism
	                      over ${\Bbb R}\hookrightarrow{\Bbb C}$}\\[.6ex]
    \Longrightarrow
      & \mbox{$\varphi^{\sharp}$ is a ring-homomorphism
	                      over ${\Bbb R}\hookrightarrow{\Bbb C}$}\;.
  \end{array}	
 $$
And $\varphi^{\sharp}$ in Definition~1.2
 is what we employed in [L-Y2] to define the notion of a $C^k$-map
 $$
   \varphi\; :\;  (X^{\!A\!z},E)\; \longrightarrow\; Y\,,
 $$
 following the spirit of Grothendieck's setting for modern (commutative) Algebraic Geometry.
It resolves
 the issue that $C^k(\End_{\Bbb C}(E))$, $r\ge 2$, can never be made a $C^k$-ring
 and at the same time makes
     the notion of the `{\it graph}' of $\varphi$,
	 a ${\cal O}_{X\times Y}^{\,\Bbb C}$-module $\tilde{\cal E}_{\varphi}$,
   naturally built into the definition of the differentiable map $\varphi$.

\bigskip

\begin{sconjecture} {\bf [$C^k$-map vs.\ ring-homomorphism].}
 Let
   $X$ and $Y$ be $C^k$-manifolds  and $E$ be a complex $C^k$ vector bundle of rank $r$ on $X$.
 Given a correspondence
   $$
     \varphi^{\sharp}\;: C^k(Y)\;\longrightarrow\; C^k(\End_{\Bbb C}(E))\,.
   $$
 Then, the following three statements are equivalent:
   \begin{itemize}
     \item[\rm (1)]
	  $\varphi^{\sharp}$ is a ring-homomorphism over ${\Bbb R}\hookrightarrow {\Bbb C}$.
	
     \item[\rm (2)]
	  $\varphi^{\sharp}$ is a weakly $C^k$-admissible ring-homomorphism
	   over ${\Bbb R}\hookrightarrow {\Bbb C}$.
	
	 \item[\rm (3)]
	  $\varphi^{\sharp}$ is a $C^k$-admissible ring-homomorphism
	   over ${\Bbb R}\hookrightarrow {\Bbb C}$.
   \end{itemize}
\end{sconjecture}

\bigskip

Thus, if justified,
 any of the $\varphi^{\sharp}$ in Statements (1), (2), or (3) above
 can be used to define the notion of a $C^k$-map $\varphi:(X^{\!A\!z},E)\rightarrow Y$;
cf.\ [L-Y2: Sec.\ 5] (D(11.1)).
The resulting notions are the same/equivalent.

\bigskip

Next, we recall the notion of `nilpotency' in three situations  and  then
 bring forth the second conjecture of the current note.

\bigskip

\begin{sdefinition}{\bf [nilpotency].} \rm
   We define the notion of {\it nilpotency} in three situations.
   \begin{itemize}
     \item[(1)]
      Let $a\in R$  be a nilpotent element of a ring (commutative or not).
	  We say $a$ has nilpotency $\le l\in{\Bbb Z}_{\ge 1}$ if $a^l=0$.
	  The minimal such $l$ is called the {\it nilpotency} of $a$.

     \item[(2)]
      Let $R$ be a ring.
	  We say that $R$ has nilpotency $\le l$ if $a^l=0$ for all nilpotent elements of $R$.
	  The minimal such $l$ is called the {\it nilpotency} of $R$.
 	
	 \item[(3)]
      Let $m\in M_{r\times r}({\Bbb C})$ be an $r\times r$-matrix with entries in ${\Bbb C}$.
	  We say that $m$ has nilpotency $\le l$ if each elementary Jordan block of $m$ 	
	 {\scriptsize
	    $$
	     \left[
		   \begin{array}{ccccc}
		     \lambda  & 1 \\
			    & \lambda  & \ddots\\
		        & & \ddots & & 1 \\
			 &&&& \lambda
		   \end{array}
		 \right]_{l^{\prime}\times l^{\prime}}\,,	
	    $$}with
	  all entries not on the diagonal nor on the first upper off-diagonal being equal to zero,
	  has the size $l^{\prime}\le l$.	
	  The minimal such $l$ is called the {\it nilpotency} of $m$. 	
   \end{itemize}
\end{sdefinition}

\medskip

\begin{sconjecture} {\bf [$C^k$-map to ${\Bbb R}^n$].}
 Let
  $X$ be a $C^k$-manifold  and $E$ be a complex $C^k$ vector bundle of rank $r$ on $X$.
 Let
  $(y^1,\,\cdots\,, y^n)$ be a global coordinate system on ${\Bbb R}^n$, as a $C^k$-manifold, and
  $$
    \eta\;:\; y^i\; \longmapsto\; m_i\,\in\, C^k(\End_{\Bbb C}(E))\,,\;\;
	i\;=\;1,\,\ldots\,,n\,,
  $$
 be an assignment  such that
  \begin{itemize}
   \item[(1)]
     $\;m_im_j\;=\;m_jm_i$, for all $i,\,j\,$;

   \item[(2)]
    for every $p\in X$,
	 the eigenvalues of the restriction
	   $m_i(p)\in \End_{\Bbb C}(E|_p)\simeq M_{r\times r}({\Bbb C})$
	   are all real;
	
   \item[(3)]
    for every $p\in X$,
	 the nilpotency of $m_i(p)$ $\le k+1$.
 \end{itemize}
 Then,
  $\eta$ extends to a unique $C^k$-admissible ring-homomorphism
  $$
    \varphi_{\eta}^{\sharp}\;:\;
	 C^k({\Bbb R}^n)\; \longrightarrow\; C^k(\End_{\Bbb C}(E))
  $$
  over ${\Bbb R}\hookrightarrow{\Bbb C}$ and. hence,
  defines a $C^k$-map $\varphi_{\eta}:(X^{\!A\!z},E)\rightarrow Y$.
\end{sconjecture}

\bigskip

Note that
  the set of conditions (1), (2), and (3) are necessary
    for $\eta$ to be extendable to a ring-homomorphism
    $C^k(Y)\rightarrow C^k(\End_{\Bbb C}(E))$;
  cf.\ [L-Y2: Sec.~3] (D(11.1)).
This conjecture rings with the fact that
 any $C^k$-map $f:Z\rightarrow {\Bbb R}^k$ from a $C^k$-manifold $Z$ to ${\Bbb R}^n$
  is specified by the component maps $f_i:=  \pr_i\circ f:Z\rightarrow {\Bbb R}$,
  where $\pr_i:{\Bbb R}^n\rightarrow {\Bbb R}$
   is the projection map to the $i$-th coordinate of ${\Bbb R}^n$, $1\le i\le n$.

\bigskip

Similarly, there are also the fermionic/super version of these conjectures,
 which would give, in particular, equivalent notions of
 $C^k$-maps from an Azumaya/matrix super $C^k$-manifold to a $C^k$-manifold or super $C^k$-manifold
 to that defined in [L-Y3: Sec.\ 4] (D(11.2)).

\bigskip

The goal of the current note is to prove Conjecture~1.3 and Conjecture~1.5 in the case $k=\infty$.
Their fermionic/super version hold similarly.
Before that, let us take a look at the preliminary case when $X$ is a point.

\bigskip

\section{Preliminaries: When $X$ is a point}	

For general $k\in {\Bbb Z}_{\ge 0}\cup\{\infty\}$,
 we examine and prove Conjecture~1.3 and Conjecture~1.5 for the special case that $X$ is a point.
For simplicity of notation,
 some of the explicit expressions in the discussion are meant for $k$ being finite;
 for example, Taylor polynomials at a point or infinitesimal neighborhoods around a point.
They can be readily converted to the case $k={\infty}$
 (by restricting the Taylor polynomial to degree $r$
     or the nilpotency of the infinitesimal neighborhood to $r$).

\bigskip

\begin{flushleft}
{\bf The canonical $C^k$-ring structure on a commutative finite-dimensional ${\Bbb R}$-algebra}
\end{flushleft}

\begin{slemma}
{\bf [canonical $C^k$-ring structure on finite-dimensional ${\Bbb R}$-algebra].}	
 Let $A$ be a commutative finite-dimensional ${\Bbb R}$-algebra of nilpotency $\le l$.
 Then, for all $k\ge l-1$,
   $A$ admits a canonical $C^k$-ring structure
   that is compatible with the underlying ring structure of $A$.
\end{slemma}

\begin{proof}
 Since $A$ factorizes into a product
   $$
     A = A_1\times\,\cdots\,\times A_s
   $$
  of Weil algebras
    (i.e.\ commutative finite-dimensional ${\Bbb R}$-algebra with a unique maximal ideal)
  that is unique up to a permutation of the factors
  and  a product of $C^k$-rings admits a canonical $C^k$-ring structure from the factors,
 without loss of generality we may assume that  $A$ is a Weil algebra.

 In this case, there is a built-in ${\Bbb R}$-algebra quotient
   $$
     \alpha\; :\;  A\; \longrightarrow\; {\Bbb R}
   $$
   whose kernel is the maximal ideal ${\frak m}$ of $A$.
 Together with the built-in inclusion ${\Bbb R}\hookrightarrow A$  for any ${\Bbb R}$-algebra,
  one has a sequence of ${\Bbb R}$-algebra homomorphisms
  $$
   \xymatrix{
    {\Bbb R}\;  \ar@{^{(}->}[r]   & \;A\; \ar[r]^-{\alpha}    & \;{\Bbb R}
	}
  $$
  with the composition the identity homomorphism on ${\Bbb R}$.
 This gives a canonical splitting
  $$
    A\;=\; {\Bbb R}\oplus {\frak m}
  $$
  as ${\Bbb R}$-vector spaces, with ${\frak m}$
      identical with the nil-ideal of $A$ of nilpotency $\le l$.

 Let $h\in C^k({\Bbb R}^n)$  for any $n\in {\Bbb Z}_{\ge 1}$.
 Then,
  for any $a_1$, $\cdots$, $a_n\in A$,
  define  $h(a_1,\,\cdots\,,\,a_n)$ by setting
  $$
    h(a_1,\,\cdots\,,\,a_n)\; =\;
	  \sum_{s=0}^k\,\frac{1}{s!}\,
	    \sum_{d_1+\,\cdots\,+d_n=s}
	     \partial_1^{\,d_1}\,\cdots\,\partial_n^{\,d_n}  h(b_1,\,\cdots\,,\, b_n)\,
		  c_1^{d_1}\,\cdots\,c_n^{d_n}\,,	
  $$
  where
    \begin{itemize}
	  \item[\LARGE $\cdot$]
       $a_i=b_i+c_i$, $i=1,\,\ldots\,,n$,
        is the decomposition of $a_i$ according to $A={\Bbb R}\oplus {\frak m}$;
		
	  \item[\LARGE $\cdot$]	
	  	$\partial_1^{\,d_1}\,\cdots\,\partial_n^{\,d_n}  h$
       is the partial derivative of $h$ with respect to the first variable $d_1$-times, the second variable $d_2$-times,
	    ..., and the $n$-th variable $d_n$-times.
	\end{itemize}	
 Notice that $(b_1,\,\cdots\,,\,b_n)\in {\Bbb R}^n$.
 Thus,
  while
    $$
      \partial_1^{\,d_1}\,\cdots\,\partial_n^{\,d_n}  h\,, \hspace{2em}
         d_1,\,\cdots\,,\,d_n\in {\Bbb Z}_{\ge 0}\,, \;\; d_1+\cdots+d_n=s\,,
    $$		
    in general lie only in $C^{k-s}({\Bbb R}^n)$,
  the evaluation
   $\partial_1^{\,d_1}\,\cdots\,\partial_n^{\,d_n}  h (b_1,\,\cdots\,,\,b_n)$
   remains well-defined,
  as is required for $h\in C^k({\Bbb R}^n)$.

 This defines the canonical $C^k$-ring structure on $A$.
 Clearly, it is compatible with the underlying ring structure of $A$.

\end{proof}

\bigskip

The following two lemmas follow by construction.
They are indications
  that the canonical $C^k$-ring structure on a finite-dimensional ${\Bbb R}$-algebra, when exists,
    is functorial/natural.

\bigskip

\begin{slemma} {\bf [$\Bbb R$-algebra homomorphism vs.\ $C^k$-ring homomorphism, I].}
 Let $A$ and $B$
   be commutative finite-dimensional ${\Bbb R}$-algebras with both of nilpotency $\le k+1$.
 Then
  $$
    \Hom_{\mbox{\scriptsize\it ${\Bbb R}$-${\cal A}$lgebras}}(B,A)\;
	 =\;   \Hom_{\mbox{\scriptsize\it $C^k$-${\cal R}$ings}}(B,A)
  $$
  with respect to the canonical $C^k$-ring structure on $A$ and $B$ respectively.
\end{slemma}

\bigskip

\begin{slemma} {\bf [$\Bbb R$-algebra homomorphism vs.\ $C^k$-ring homomorphism, II].}
 Let
   $A$ be a commutative finite-dimensional ${\Bbb R}$-algebra of nilpotency $\le k+1$  and
   $Y$ be a $C^k$-manifold.
 Then
  $$
    \Hom_{\mbox{\scriptsize\it ${\Bbb R}$-${\cal A}$lgebras}}(C^k(Y),A)\;
	 =\;   \Hom_{\mbox{\scriptsize\it $C^k$-${\cal R}$ings}}(C^k(Y),A)
  $$
  with respect to the canonical $C^k$-ring structure on $A$.
\end{slemma}

\bigskip

\begin{flushleft}
{\bf Validity of Conjecture~1.3 when $X$ is a point}    
\end{flushleft}
The following lemma follows from [L-Y2: Sec.\ 3] (D(11.1)):

\bigskip

\begin{slemma}{\bf [ring-homomorphism to matrix algebra].}
  Given a $C^k$-manifold $Y$,
   let
     $$
	    \varphi^{\sharp}\; :\;  C^k(Y)\; \longrightarrow\;  M_{r\times r}({\Bbb C})
     $$
     be a ring-homomorphism over ${\Bbb R}\hookrightarrow{\Bbb C}$.
  Then
   \begin{itemize}
     \item[(1)]
	   For all $f\in C^k(Y)$, $\varphi^{\sharp}(f)\in M_{r\times r}({\Bbb C})$
	   has only real eigenvalues.
	
	 \item[(2)]
      $\varphi^{\sharp}$ factors through a finite-Taylor-expansion map
	  at a finite set $\{q_1,\,\cdots\,,\,q_s\}$ of $Y$
      $$
	    \xymatrix{
	     C^k(Y) \ar[rr]^-{\varphi^{\sharp}}
	           \ar[d]_-{\oplus_{j=1}^sT_{q_j}^{(k)} }
	       && M_{r\times r}({\Bbb C})    \\
         \oplus_{j=1}^s
	          \frac{{\Bbb R}[y_j^1,\,\cdots\,, y_j^n] }
	                   {(y_j^1,\,\cdots\,,\, y_j^n)^{k+1}}
	 	    	\ar[rru]_-{\underline{\varphi}^{\sharp}}    &&&,
	    }
       $$	
  	   where
	   \begin{itemize}
	    \item[$\cdot$]
		 $(y_j^1,\,\cdots\,,\, y_j^n)$ is a local coordinate system in a neighborhood of $q_i\in Y$
		   with the coordinates of $q_j$ all $0$,
		
		\item[$\cdot$]
		 $T_{q_j}^{(k)}$ is the map
		    `taking Taylor polynomial (of elements in $C^k(Y)$) at $q_j$
		      with respect to $(y_j^1,\,\cdots\,,\, y_j^n)$ up to and including degree $k$', and
			
	    \item[$\cdot$]		
		 $\underline{\varphi}^{\sharp}$ is an (algebraic) ring-homomorphism
		 over ${\Bbb R}\subset {\Bbb C}$.		
	   \end{itemize}
	 	
	 \item[(3)]
      Nilpotency of $\varphi^{\sharp}(C^k(Y))$ is bounded by 	$\min\{k+1, r\}$. 	
   \end{itemize}
\end{slemma}

\bigskip

\noindent
See {\sc Figure}~2-1 for the contravariant geometry behind;
cf.\ [L-Y2: {\sc Figure}~3-4-1] (D(11.1)).
%

 \begin{figure} [htbp]
  \bigskip
  \centering

  \includegraphics[width=0.60\textwidth]{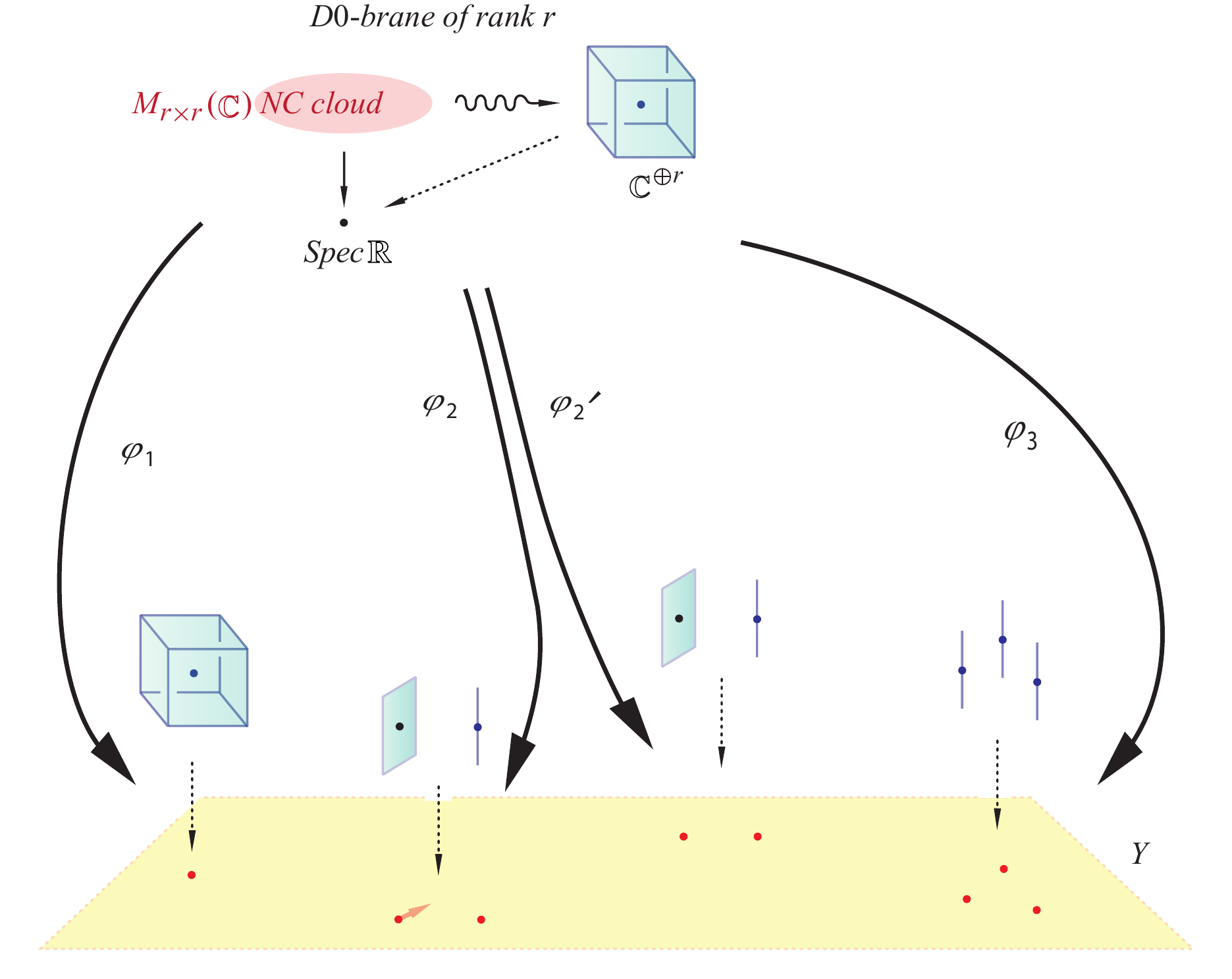}

  \bigskip
  \centerline{\parbox{13cm}{\small\baselineskip 12pt
   {\sc Figure}~2-1.
    Four examples of $C^k$ maps
     $\varphi:(p^{A\!z},{\Bbb C}^{\oplus r})\rightarrow Y$
      from an Azumaya/matrix point with a fundamental module to a $C^k$-manifold $Y$
	  are illustrated.
	The nilpotency of the image scheme $\Image\varphi$ in $Y$ is bounded by $\min\{k+1, r\}$.	
    In the figure, the push-forward of the fundamental module in each example is also indicated.
       }}
  \bigskip
 \end{figure}	

\bigskip

It follows from Lemma~2.1 that
 both
   $\oplus_{j=1}^s
	      \frac{{\Bbb R}[y_j^1,\,\cdots\,, y_j^n] }
	                  {(y_j^1,\,\cdots\,,\, y_j^n)^{k+1}}$  and
   $A_{\varphi}:= \Image(\varphi^{\sharp})\subset M_{r\times r}({\Bbb C})$
   admit a canonical $C^k$-ring structure that is compatible with their underlying ring structure.
In terms of this,
 both ring-homomorphisms $\oplus_{j=1}^sT_{q_j}^{(k)}$ and
  $\underline{\varphi}^{\sharp}$, regarded now as a ring-morphism to $A_{\varphi}$,
  are also $C^k$-ring homomorphisms.
Thus,

\bigskip

\begin{scorollary} {\bf [$\varphi^{\sharp}$ $C^k$-admissible].}
   As a ring-homomorphism to $A_{\varphi}$,
     $\varphi^{\sharp}: C^k(Y)\rightarrow A_{\varphi}$  is a $C^k$-ring homomorphism.
   Thus, as a ring-homomorphism to $M_{r\times r}({\Bbb C})$,
     $\varphi^{\sharp}: C^k(Y)\rightarrow M_{r\times r}({\Bbb C})$
	 is $C^k$-admissible.
\end{scorollary}

\bigskip

This proves Conjecture~1.3 when $X$ is a point.

\bigskip

\begin{flushleft}
{\bf Validity of Conjecture~1.5 when $X$ is a point}  
\end{flushleft}
Given ${\Bbb R}^n$, as a $C^k$-manifold, with coordinate $(y^1\,\cdots\,, y^n)$
 and an assignment
  $$
    \eta\;:\; y^i\; \longmapsto\; m_i\,\in\,  M_{r\times r}({\Bbb C})\,,\;\;
	i\;=\;1,\,\ldots\,,n\,,
  $$
 that satisfies
  \begin{itemize}
   \item[(1)]
     $\;m_im_j\;=\;m_jm_i$, for all $i,\,j\,$;

   \item[(2)]
	 the eigenvalues of $m_i$ are all real;
	
   \item[(3)]
	 the nilpotency of $m_i$ $\le k+1$, for all $i$.
  \end{itemize}	
Then
  Properties (1) and (2) together imply that
  the collection $\{m_1,\,\cdots\,,\, m_n\}$ of matrices are simultaneously triangularizable
   $$
     m_i\;\sim\;
	  \left[
        \begin{array}{ccc}
		     \lambda^i_1  & \ast  & \ast \\
			    & \ddots  & \ast    \\
		0	 && \lambda^i_r
		\end{array}
	  \right]_{r\times r}\,,	
  $$
  with $\lambda^i_j\in {\Bbb R}$, $1\le i \le n$, $1\le j\le r$.
 Let
  $$
    q_j\; =\; (\lambda_j^1,\,\cdots\,,\,\lambda_j^n)\; \in \; {\Bbb R}^n\,.
  $$
 Then, after removing repetitive copies and relabelling,
  the finite set of points $q_1$, $\cdots$, $q_s$, for some $s\le r$, in ${\Bbb R}^n$
   is an invariant of  the commuting tuple $(m_1, \,\cdots\,,\, m_n)$
   of matrices in $M_{r\times r}({\Bbb C})$.
 Furthermore,
  the fundamental representation ${\Bbb C}^{\oplus n}$ of $M_{r\times r}({\Bbb C})$
  decomposes into the direct sum
  $$
   {\Bbb C}^{\oplus n}\;=\; V_1\oplus\, \cdots\, \oplus V_s
  $$
  of $s$-many common invariant subspaces of $m_1$, $\cdots$, $m_n$
  such that
  the Jordan form of
  $$
     m_i|_{V_j}\; :\;  V_j\; \longrightarrow\;  V_j
  $$
   has diagonal entries all equal to $\lambda^i_j$.
 After a change of basis of ${\Bbb C}^{\oplus n}$ and for  simplicity of notation,
  we may assume that the decomposition is given by
  $$
    {\Bbb C}^{\oplus n}\;
	 =\; {\Bbb C}^{\oplus r_1}\oplus\, \cdots\, \oplus {\Bbb C}^{\oplus r_s}\,,
  $$
  with the $j$-th summand ${\Bbb C}^{\oplus r_j }$
   being associated to $(\lambda^1_j, \,\cdots\,,\,\lambda^n_j)$.
 In terms of this decomposition,
  one can re-express $\eta$ as an assignment
  $$
    \eta\;:\; y^i\; \longmapsto\;
	                  (m_{i, 1},, \,\cdots\,,\, m_{i,s})\;
					  \in\, M_{r_1\times r_1}({\Bbb C})
					                 \times \,\cdots\,\times M_{r_s\times r_s}({\Bbb C})\;
                 \subset\; M_{r\times r}({\Bbb C})\,.									
  $$
 In this expression,
   it is immediate that
   $\eta$ extends to a ring-homomorphism
   $$
     \varphi_{\eta}^{\sharp}\;  :\;
	   C^k({\Bbb R}^n)\; \longrightarrow\; M_{r\times r}({\Bbb C})
   $$
   from the composition of ring-homomorphisms
    $$
	  \xymatrix{
	     C^k({\Bbb R}^n)
	           \ar[d]_-{\oplus_{j=1}^sT_{q_j}^{(k)} }
	       && \hspace{2em}
		         M_{r_1\times r_1}\times \,\cdots\,\;\times M_{r_s\times r_s}({\Bbb C})\;
		           \subset\; M_{r\times r}({\Bbb C})    \\
         \oplus_{j=1}^s
	          \frac{{\Bbb R}[y^1-\lambda_j^1,\,\cdots\,, y^n-\lambda_j^n] }
	                       {(y^1-\lambda_j^1,\,\cdots\,,\, y^n-\lambda_j^n)^{k+1}}
	 	    	\ar[rru]_-{\underline{\varphi}^{\sharp}\;
				                       =\; (\, \underline{\varphi}^{\sharp}_1,\,\cdots\,,\,
        									       \underline{\varphi}^{\sharp}_s\,) }    &&&,
	  }
    $$	
  where
   \begin{itemize}
    \item[\LARGE $\cdot$]
	  $T_{q_j}^{(k)}$ is the map
		 `taking Taylor polynomial (of elements in $C^k({\Bbb R}^n)$) at $q_j$
		      with respect to coordinate $(y^1,\,\cdots\,,\, y^n)$ up to and including degree $k$', and
	
    \item[\LARGE $\cdot$]
      $$
	    \underline{\varphi}^{\sharp}_j\; :\;  		
		  	   \frac{{\Bbb R}[y^1-\lambda_j^1,\,\cdots\,, y^n-\lambda_j^n]}
	                    {(y^1-\lambda_j^1,\,\cdots\,,\, y^n-\lambda_j^n)^{k+1}}\;
          \longrightarrow\; M_{r_j\times r_j}({\Bbb C})\,,							
	  $$ 	
	  is the ${\Bbb R}$-algebra homomorphism generated by sending
	   $y^i\mapsto m_{i,j}$, $i=1,\,\ldots\,,\,n$.
   \end{itemize}

 Equip
    $\frac{{\Bbb R}[y^1-\lambda_j^1,\,\cdots\,, y^n-\lambda_j^n]}
	                    {(y^1-\lambda_j^1,\,\cdots\,,\, y^n-\lambda_j^n)^{k+1}}$   and
    $\Image(\underline{\varphi}^{\sharp}_j)$						
 with the canonical $C^k$-ring structure.
 Then all of
    $T_{q_j}^{(k)}$  and
	$\underline{\varphi}^{\sharp}_j$, $j=1,\,\ldots\,,\,s$,
  become $C^k$-ring homomorphisms.	
 Let $A_{\varphi_{\eta}}:= \Image (\varphi_{\eta}^{\sharp})$
   be equipped with the canonical $C^k$-ring structure.
 It follows then that
  the ring-homomorphism
  $$
     \varphi^{\sharp}_{\eta}\;:\;
	   C^k({\Bbb R}^n)\; \longrightarrow\; A_{\varphi_{\eta}}
  $$
 is also a $C^k$-ring homomorphism.
 This proves Conjecture~1.5 when $X$ is a point.

\bigskip

\section{Proof of Conjectures in the $C^{\infty}$ case}

Conjecture~1.3 and Conjecture~1.5 in the $C^{\infty}$ case are examined in this section for general $X$.

\bigskip

\subsection{Proof of Conjecture~1.3 in the $C^{\infty}$ case}

We prove in this subsection Conjecture~1.3 in the $C^{\infty}$ case:
			
\bigskip

\begin{theorem} {\bf [$C^{\infty}$-map vs.\ ring-homomorphism].}
 Let
   $X$ and $Y$ be $C^{\infty}$-manifolds  and
   $E$ be a complex $C^{\infty}$ vector bundle of rank $r$ on $X$.
 Given a correspondence
   $$
     \varphi^{\sharp}\;:
	   C^{\infty}(Y)\; \longrightarrow\; C^{\infty}(\End_{\Bbb C}(E))\,.
   $$
 Then, the following three statements are equivalent:
   \begin{itemize}
     \item[\rm (1)]
	  $\varphi^{\sharp}$ is a ring-homomorphism over ${\Bbb R}\hookrightarrow {\Bbb C}$.
	
     \item[\rm (2)]
	  $\varphi^{\sharp}$ is a weakly $C^{\infty}$-admissible ring-homomorphism
	   over ${\Bbb R}\hookrightarrow {\Bbb C}$.
	
	 \item[\rm (3)]
	  $\varphi^{\sharp}$ is a $C^{\infty}$-admissible ring-homomorphism
	   over ${\Bbb R}\hookrightarrow {\Bbb C}$.
   \end{itemize}
\end{theorem}

\medskip

\begin{proof}
 Since Statement(3) $\Rightarrow$ Statement (2) $\Rightarrow$ Statement (1),
  one only needs to show that Statement (1) $\Rightarrow$ Statement (3).

 \bigskip

 \noindent
 {\it Step $(a):$ The only natural candidate extension}

 \medskip

 \noindent
 Let $\varphi^{\sharp}: C^{\infty}(Y)\rightarrow C^{\infty}(\End_{\Bbb C}(E))$
   be a ring-homomorphism over ${\Bbb R}\hookrightarrow {\Bbb C}$.
 Consider the ${\Bbb C}$-algebra $C^{-\infty}(\End_{\Bbb C}(E))$
   of sections of the endomorphism bundle
   $\End_{\Bbb C}(E)\rightarrow X$ as a map between sets.
 Then $\varphi^{\sharp}$ extends canonically to a ring-homomorphism
  $$
   \xymatrix{
     C^{-\infty}(\End_{\Bbb C}(E))
	  & \;C^{\infty}(\End_{\Bbb C}(E)) \ar@{_{(}->}[l]
	  &&  C^{\infty}(Y)\ar[ll]_-{\varphi^{\sharp}} \ar@{_{(}->}[d]  \\	
    &&&  C^{\infty}(X\times Y)\ar[lllu]^-{\tilde{\varphi}^{\sharp}}
	}
  $$
  over ${\Bbb C}\hookleftarrow{\Bbb R}$,
    where both inclusions in the diagram are tautological,
  as follows:
 \begin{itemize}
  \item[{\Large $\cdot$}]
    Associated to each $f\in C^{\infty}(X\times Y)$ is the subset
 	$$
	   X^{f;1}\; := \;
	       \{(p, f|_{\{p\}\times Y})\,:\, p\in X\}\; \subset\; X\times C^{\infty}(Y)\,.
    $$
	
  \item[{\Large $\cdot$}]	
   The map
      $\Id_X\times \varphi^{\sharp}:
	       X\times C^{\infty}(Y)\rightarrow U\times C^{\infty}(\End_{\Bbb C}(E))$
	sends $X^{f;1}$ to the subset
	$$
	  X^{f;2}\;=\;
	     \{(p,   \varphi^{\sharp}(f|_{\{p\}\times Y}))\,:\, p\in X\}\;
     		 \subset\; X \times C^{\infty}(\End_{\Bbb C}(E))\,.
	$$
 	
  \item[{\Large $\cdot$}]	
    Which produces a section of $\End_{\Bbb C}(E)\rightarrow X$ as a map between sets:
	 $$
	   s_f\; =\;
	      \{(p,
		        (\varphi^{\sharp}(f|_{\{p\}\times Y}))|_p)\,:\,
				  p\in X\}\;
     		 \in\; C^{-\infty}(\End_{\Bbb C}(E))\,.
     $$
	
  \item[{\Large $\cdot$}]
   $\tilde{\varphi}^{\sharp}:
      C^{\infty}(X\times Y)\rightarrow C^{-\infty}(\End_{\Bbb C}(E))\,$
     is now defined by $\,f\mapsto s_f\,$.
 \end{itemize}
By construction,
 $\tilde{\varphi}^{\sharp}$  is a ring-homomorphism
   over ${\Bbb R}\hookrightarrow{\Bbb C}$   and
 it makes the following diagram of ring-homomorphisms commute:
 $$
   \xymatrix{
    \hspace{1ex}
     C^{-\infty}(\End_{\Bbb C}(E))
	     &  \hspace{1ex}C^{\infty}(\End_{\Bbb C}(E))\ar @{_{(}->}[l]
	   &&& \rule{0ex}{3ex}\hspace{1ex}
	            C^{\infty}(Y)\ar[lll]_-{\varphi^{\sharp}}
	                            \ar @{_{(}->}[d]^{pr_Y^{\sharp}}        \\	
    & \rule{0ex}{2.4ex}
	     \;C^{\infty}(X)\; \ar @{^{(}->}[rrr]_-{pr_X^{\sharp}} \ar@{^{(}->}[u]
	    &&& C^{\infty}(X\times Y) 	  \ar @/^1ex /  [ullll]_(.4){\tilde{\varphi}^{\sharp}}\;,
	}
  $$
  where $\pr_X: X\times Y\rightarrow X$ and $\pr_Y: X\times Y\rightarrow Y$
    are the projection maps and
  $C^{\infty}(X)\hookrightarrow C^{\infty}(\End_{\Bbb C}(E))$
    follows form the inclusion of the center $C^{\infty}(X)^{\Bbb C}$
	of $C^{\infty}(\End_{\Bbb C}(E))$.
 (Cf.\ The construction in
               [L-Y2: Sec.\ 5.1,
			    theme `{\sl A generalization to ring-homomorphisms to Azumaya/matrix algebras}'] (D(11.1)).)
 Notice that $\tilde{\varphi}^{\sharp}$ is the only extension of $\varphi^{\sharp}$
  to $C^{\infty}(X\times Y)$
  that satisfies the above commutative diagram and the natural condition that
  $$
    \tilde{\varphi}^{\sharp}|_{{p}\times Y}\;=\;
	\varphi^{\sharp}|_{p}\;:\; C^{\infty}(Y)\;\longrightarrow\; \End_{\Bbb C}(E|_p)\,,
  $$
  for all $p\in X$.

 \bigskip

 \noindent
 {\it Step $(b):$ From the aspect of germs over $X$}

 \medskip

 \noindent
 {To} understand whether $\tilde{\varphi}^{\sharp}$ takes its values in
   $C^{\infty}(\End_{\Bbb C}(E))$, one needs to know
   how $\tilde{\varphi}^{\sharp}|_{p\times Y}:
                      C^{\infty}(Y)\rightarrow \End_{\Bbb C}(E|_p)$
	varies as $p$ varies along $X$.
 This leads us to studying the germs-over-$X$ aspect of $\tilde{\varphi}^{\sharp}$,
  which we now proceed.

 \bigskip

 \noindent
 {\bf Definition~3.1.1.1. [spectral locus/subscheme $\varphi^{\sharp}$].}
   Let $I_{\varphi}\subset C^{\infty}(X\times Y)$ be the ideal of $C^{\infty}(X\times Y)$
     generated by the set
     $$
        \{\,\determinant (f\cdot\Id_{r\times r}-\varphi^{\sharp}(f)) \,|\, f\in C^{\infty}(Y)\,\}
     $$
     of elements in $C^{\infty}(X\times Y)$, where $\Id_{r\times r}$ is the $r\times r$ identity matrix.
   $I_{\varphi}$  defines a $C^{\infty}$-subscheme  $\varSigma_{\varphi}$ of $X\times Y$,
    called interchangeably the {\it spectral locus} or the {\it spectral subscheme} of $\varphi^{\sharp}$
    in $X\times Y$.

 \bigskip

 \noindent
 Notice that
  while the local matrix presentation of $\varphi^{\sharp}(f)$
    depends on the local trivialization of $E$ chosen,
  the determinant $\determinant(f\cdot\Id_{r\times r}-\varphi^{\sharp}(f))$ does not
     and, hence, is well-defined.

 Some properties of $\varSigma_{\varphi}$ that follow immediately from the defining ideal $I_{\varphi}$
  are listed below:
  \begin{itemize}
   \item[\LARGE $\cdot$]
    {\it $\varSigma_{\varphi}$ is finite over $X$}
 	 in the sense that, for all $p\in X$,
	 the preimage $\pr_X^{-1}(p)$  of the morphism $\pr_X:\varSigma_{\varphi}\rightarrow X$
	   from the restriction of the projection map $X\times Y\rightarrow Y$
	  are all $0$-dimensional $C^{\infty}$-scheme
	  with the function-ring given by a (commutative) finite-dimensional ${\Bbb R}$-algebra. 	
	
   \item[\LARGE $\cdot$]
    A comparison with the study of ring-homomorphisms from $C^{\infty}({\Bbb R}^n)$ to
	 $M_{r\times r}({\Bbb C})$ in [L-Y2: Sec.\ 3.2] (D(11.1))
	 implies that
	 \begin{itemize}
	  \item[-$\;$]
       $\tilde{\varphi}^{\sharp}(I_{\varphi})\;=\; 0\,$.
	
	  \item[-$\;$]\it
	   for all $f\in C^{\infty}(X\times Y)$,
       $\tilde{\varphi}^{\sharp}(f)\in C^{-{\infty}}(\End_{\Bbb C}(E))$
	   depends only on the restriction of $f$ on the $C^{\infty}$-subscheme
	    $\varSigma_{\varphi}\subset X\times Y$.	
	 \end{itemize}
  \end{itemize} 	
 We emphasize that,
  being a $C^{\infty}$-scheme defined by an ideal of $C^{\infty}(X\times Y)$,
  the spectral locus $\varSigma_{\varphi}$ of $\varphi$
    is more than just a closed subset of $X\times Y$;
  cf.\ {\sc Figure}~3-1-1-1.

 \begin{figure} [htbp]
  \bigskip
  \centering

  \includegraphics[width=0.55\textwidth]{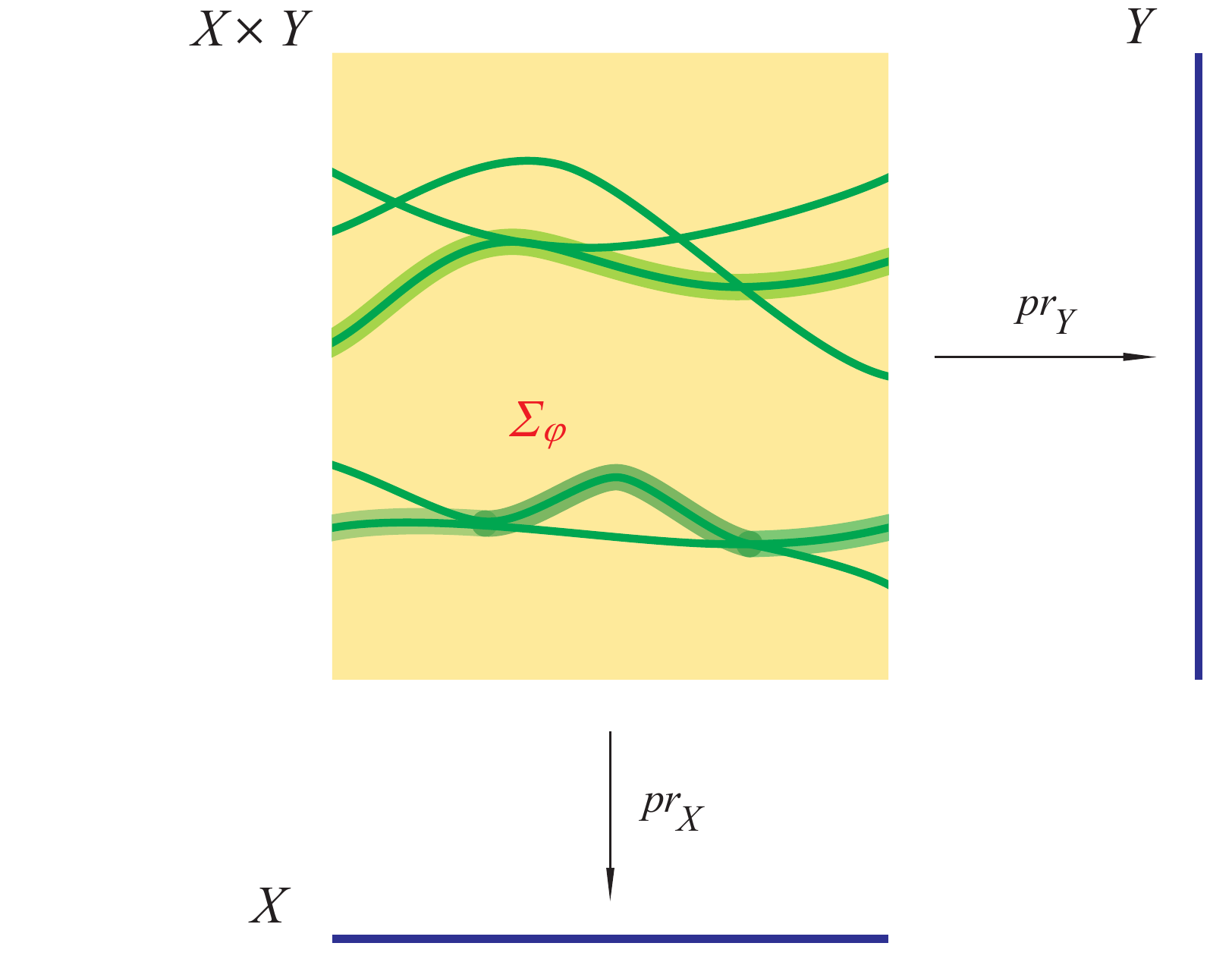}

  \bigskip
  \bigskip
  \centerline{\parbox{13cm}{\small\baselineskip 12pt
   {\sc Figure}~3-1-1-1.
     The spectral subscheme $\varSigma_{\varphi}$
	   (in green color, with the green shade indicating the nilpotent structure/cloud
	       on $\varSigma_{\varphi}$)
	   in $X\times Y$
	  associated to a ring-homomorphism
	   $\varphi^{\sharp}:C^{\infty}(Y)\rightarrow C^{\infty}(\End_{\Bbb C}(E))$.
	 More than just a point-set with topology, 
	   it is a $C^{\infty}$-scheme that is finite over $X$.
       }}
  \bigskip
 \end{figure}

 Recall the morphism $\pr_X:\varSigma_{\varphi}\rightarrow X$.
 Let $p\in X$.
 Then
  since $\pr_X^{-1}(p)$ is $0$-dimensional,
  there exists an open neighborhood $U$ of $p$ such that
   $\pr_X^{-1}(U)$
     is contained in an open subset $U\times V$ of $X\times Y$,
	 where $V$ is an open subset of $Y$ that is diffeomorphic to ${\Bbb R}^n$ with $n=\dimm Y$.
 Under the diffeomorphism $V\simeq {\Bbb R}^n$,
  let  $(y^1,\,\cdots\,,\, y^n)$  be coordinates on $V$ and let
  $$
    (\pr_X^{-1}(p))_{red}\;=\; \{q_1,\,\cdots\,,\,q_s\}
  $$
  be the set of closed points in $\pr_X^{-1}(p)\subset \varSigma_{\varphi}$.
 (For notation,
     $q_j=(p; q_j^1,\,\cdots\,,\,q_j^n)\in U\times V$
	 in the coordinate system $(y^1,\,\cdots\,,\,y^n)$ on $V$.)
 Consider the auxiliary $C^{\infty}$-subscheme
  $$
    \varSigma_{(y^1,\,\cdots\,,\,y^n)}\; \subset \; U\times V
  $$
  defined by the ideal
  $$
   \mbox{
    $I_{(y^1,\,\cdots\,,\,y^n)}\;
	 :=\;  (g_1,\,\cdots\,, g_n)\;\subset\; C^{\infty}(U\times V)\,,\;\;$
	 where
     $\;g_i\;:=\;  \determinant(y^i\cdot\Id_{r\times r}-\varphi^{\sharp}(y^i))\,$.
	}
  $$
 Then,
  $$
    \varSigma_{\varphi}\,\cap\, U\times V\; \subset\;  \varSigma_{(y^1,\,\cdots\,,\,y^n)}
  $$
  and, again, one has
  $$
    \tilde{\varphi}^{\sharp}(I_{(y^1,\,\cdots\,,\,y^n)})\; =\; 0\,.
  $$

  Now let
  $$
    d_{i,1},\,\cdots\,,\,  d_{i,s}
  $$
  be the regularity of $g_i$ along the $y^i$-coordinate direction at $q_1,\,\cdots,\,\, q_s$ respectively
  (cf.\ [Br: 6.1 Definition]).
  I.e.\
  $$
    g_i(q_j)\;=\; \partial_ig_i(q_j) \; =\; \cdots\;
	 =\; \partial_i^{d_{i,j}-1}g_i(q_j)\;=\; 0
	\hspace{1em}\mbox{while}\hspace{1em}
	\partial_i^{d_{i,j}}g_i(q_j)\;\ne \; 0\,.	
  $$
  Here, $\partial_i := \partial/\partial y^i$.
 Then, it follows from the Malgrange Division Theorem ([Mal]; see also [Br], [Mat1], [Mat2], [Ni]) that
  \begin{itemize}
    \item[]\it
    the germ of $f\in C^{\infty}(X\times Y)$ at $q_j$ admits a normal form
	 $$
	    f= f_0^{(q_j)}\,+\, f_1^{(q_j)}
	 $$
	with
	 $$
	  \begin{array}{c}
	    f_0^{(q_j)}\; \in\; C^{\infty}(U)[y^1,\,\cdots\,,\, y^n]
        \hspace{1em}\mbox{of $\;(y^1,\,\cdots\,,\,y^n)$-degree}\;
        \le(d_{1,j}-1,\,\cdots\,,\, d_{n,j}-1)\\[2ex]
       \hspace{2em}\mbox{and}\hspace{2em}
       f_1^{(q_j)}\;\in \; I_{(y^1,\,\cdots\,,\,y^n)}\,. \hspace{5em}	   		
	  \end{array}
	 $$
  \end{itemize}
 After
   shrinking the neighborhood $U$ of $p\in X$ further, if necessary,  and
   capping $f_0^{(q_j)}$ (still denoted by $f_0^{(q_j)}$)
  by a smooth cutoff function with support a disjoint union of small enough coordinate balls
  around $q_j$, $j=1,\,\ldots\,,\,s$,
 $$
  \tilde{\varphi}^{\sharp}(f)|_{U}\;
   =\;   \tilde{\varphi}^{\sharp}(\,\mbox{$\sum_{j=1}^s f_0^{(q_j)}$}\,)\;
   \in\; C^{\infty}(X)[\varphi^{\sharp}(y_1),\,\cdots\,,\, \varphi^{\sharp}(y^n)]\;
   \subset\; C^{\infty}(\End_{\Bbb C}(E|_U))
 $$
 since
   \begin{itemize}
    \item[\LARGE $\cdot$]
	 $\tilde{\varphi}^{\sharp}(h)\;=\; \varphi^{\sharp}(h) $ for all $h\in C^{\infty}(Y)$,
	   and particularly for $y^i$, $i=1\,\ldots\,,\,n\,$;
	\item[\LARGE $\cdot$]
	 $\tilde{\varphi}^{\sharp}(h)\;=\; h\cdot \Id_{E}$ for all $h\in C^{\infty}(X)$,
	  where $\Id_E$ is the identity map on $E$.
   \end{itemize}
 Since smoothness is a local (indeed, infinitely infinitesimal) property, 	
   smoothness of $\tilde{\varphi}^{\sharp}(f)$  for all $f\in C^{\infty}(X\times Y)$
   follows.
 This shows that
	 $\Image(\tilde{\varphi}^{\sharp})\subset C^{\infty}(\End_{\Bbb C}(E))$.

  \bigskip

  \noindent
  {\it Step $(c):$ Conclusion}

  \medskip

  \noindent
 Let $A_{\varphi}:=\Image(\tilde{\varphi}^{\sharp})$,
  which is identical to the $C^{\infty}(X)$-subalgebra
    $C^{\infty}\langle \Image(\varphi^{\sharp}) \rangle$
	 of $C^{\infty}(\End_{\Bbb C}(E))$
	 generated by $C^{\infty}(X)$ and $\Image(\varphi^{\sharp})$
	 in $C^{\infty}(\End_{\Bbb C}(E))$.
 Then, 			
   in the $C^{\infty}$ case,  as a consequence of the Hadamard's Lemma,
   the $C^{\infty}$-ring structure on $C^{\infty}(X\times Y)$ always descends,
     via $\tilde{\varphi}^{\sharp}$,
   to a $C^{\infty}$-ring structure on $A_{\varphi}$
    that is compatible with the underlying ring-structure of $A_{\varphi}$.
  In this way, one obtains a commutative diagram
     $$
		\xymatrix{
	       A_{\varphi}
			     &&& C^{\infty}(Y) \ar[lll]_-{\varphi^{\sharp}}
			                                      \ar@{_{(}->}^-{pr_Y^{\sharp}}[d]   \\			    
			   \rule{0ex}{1em}C^{\infty}(X) \ar@{^{(}->}[u]
			                                                                 \ar@{^{(}->}[rrr]_-{pr_X^{\sharp}}
				 &&& C^{\infty}(X\times Y) \ar@{->>}[lllu]_-{\tilde{\varphi}^{\sharp}}		
		}
	 $$
    of $C^{\infty}$-ring homomorphisms.
 This shows that $\varphi^{\sharp}$ is $C^{\infty}$-admissible and proves the theorem.

\end{proof}

\bigskip

\subsection{Proof of Conjecture~1.5 in the $C^{\infty}$ case}

We prove in this subsection Conjecture~1.5 in the $C^{\infty}$ case:

\bigskip

\begin{theorem} {\bf [$C^{\infty}$-map to ${\Bbb R}^n$].}
 Let
  $X$ be a $C^{\infty}$-manifold  and
  $E$ be a complex $C^{\infty}$ vector bundle of rank $r$ on $X$.
 Let
  $(y^1,\,\cdots\,, y^n)$ be a global coordinate system on ${\Bbb R}^n$,
  as a $C^{\infty}$-manifold, and
  $$
    \eta\;:\; y^i\; \longmapsto\; m_i\,\in\, C^{\infty}(\End_{\Bbb C}(E))\,,\;\;
	i\;=\;1,\,\ldots\,,n\,,
  $$
 be an assignment  such that
  \begin{itemize}
   \item[(1)]
     $\;m_im_j\;=\;m_jm_i$, for all $i,\,j\,$;

   \item[(2)]
    for every $p\in X$,
	 the eigenvalues of the restriction
	   $m_i(p)\in \End_{\Bbb C}(E|_p)\simeq M_{r\times r}({\Bbb C})$
	   are all real.
 \end{itemize}
 Then,
  $\eta$ extends to a unique $C^{\infty}$-admissible ring-homomorphism
  $$
    \varphi_{\eta}^{\sharp}\;:\;
	 C^{\infty}({\Bbb R}^n)\; \longrightarrow\; C^{\infty}(\End_{\Bbb C}(E))
  $$
  over ${\Bbb R}\hookrightarrow{\Bbb C}$ and, hence,
  defines a $C^{\infty}$-map $\varphi_{\eta}:(X^{\!A\!z},E)\rightarrow {\Bbb R}^n$.
\end{theorem}

\bigskip

Notice that Condition (3): {\it for every $p\in X$, the nilpotency of $m_i(p)$ $\le k+1$}
 in the statement of Conjecture~1.5 is automatically satisfied in the $C^{\infty}$ case.

\bigskip

\begin{proof}
 Given $\eta$ in the statement of the theorem,
  it follows from Sec.~2 that
    for all $p\in X$, the assignment from restriction
    $$
     \eta_p\;:\; y^i\; \longmapsto\; m_i(p)\,\in\,  \End_{\Bbb C}(E|_p)\,,\;\;
	  i\;=\;1,\,\ldots\,,n\,,
    $$	
    extends uniquely to a ring-homomorphism
    $$
      \varphi^{\sharp}_{\eta_p}\;:\;
	   C^{\infty}({\Bbb R}^n)\; \longrightarrow\;   \End_{\Bbb C}(E|_p)
    $$
    over ${\Bbb R}\hookrightarrow {\Bbb C}$ that is $C^{\infty}$-admissible over $p$.
 As $p$ varies,  $\eta$ extends uniquely to a ring-homomorphism
    $$
     \varphi_{\eta}^{\sharp}\;:\;
  	  C^{\infty}({\Bbb R}^n)\; \longrightarrow\;
	  C^{-{\infty}}(\End_{\Bbb C}(E))
    $$
    over ${\Bbb R}\hookrightarrow {\Bbb C}$.
 The same construction as Step (a) in the proof of Theorem~3.1.1
   extends $\varphi_{\eta}^{\sharp}$ further and uniquely to a ring-homomorphism
     $$
	    \tilde{\varphi}_{\eta}^{\sharp}\;:\;
	      C^{\infty}(X\times {\Bbb R}^n)\; \longrightarrow\;
		  C^{-{\infty}}(\End_{\Bbb C}(E))
	 $$
	 over ${\Bbb R}\hookrightarrow {\Bbb C}$
	that fits into the following commutative diagram
	 $$
      \xymatrix{
        \hspace{1ex}
         C^{-\infty}(\End_{\Bbb C}(E))
	       &  \hspace{1ex}C^{\infty}(\End_{\Bbb C}(E))\ar @{_{(}->}[l]
	      &&& \rule{0ex}{3ex}\hspace{1ex}
	               C^{\infty}({\Bbb R}^n)\ar@/_2em/[llll]_-{\varphi_{\eta}^{\sharp}}
	                               \ar @{_{(}->}[d]^{pr_{{\Bbb R}^n}^{\sharp}}        \\	
       & \rule{0ex}{2.4ex}
	        \;C^{\infty}(X)\; \ar @{^{(}->}[rrr]_-{pr_X^{\sharp}} \ar@{^{(}->}[u]
	       &&& C^{\infty}(X\times {\Bbb R}^n) 	
		                    \ar @/^1ex /  [ullll]_(.4){\tilde{\varphi}_{\eta}^{\sharp}}\;,
	   }
     $$
	of ring-homomorphisms while satisfying the condition that
	 $$
        \tilde{\varphi}_{\eta}^{\sharp}|_{{p}\times Y}\;=\;
	    \varphi_{\eta_p}^{\sharp}\;
		:\;  C^{\infty}({\Bbb R}^n)\;\longrightarrow\; \End_{\Bbb C}(E|_p)\,,
     $$
     for all $p\in X$.
	
 The same argument as Step (b) in the proof of Theorem~3.1.1
    implies that indeed
     $\tilde{\varphi}_{\eta}^{\sharp}$
   	 takes values in $C^{\infty}(\End_{\Bbb C}(E))$.
 Thus, so does $\varphi_{\eta}^{\sharp}$.	
 As in Step (c) there,
 one thus has the following commutative diagram
     $$
		\xymatrix{
	       A_{\varphi_{\eta}}
			     &&& C^{\infty}({\Bbb R}^n) \ar[lll]_-{\varphi_{\eta}^{\sharp}}
			                                      \ar@{_{(}->}^-{pr_{{\Bbb R}^n}^{\sharp}}[d]   \\
			   \rule{0ex}{1em}C^{\infty}(X) \ar@{^{(}->}[u]
			                                                                 \ar@{^{(}->}[rrr]_-{pr_X^{\sharp}}
				 &&& C^{\infty}(X\times {\Bbb R}^n)
				                                  \ar@{->>}[lllu]_-{\tilde{\varphi}_{\eta}^{\sharp}}		
		}
	 $$
    of $C^{\infty}$-ring homomorphisms,
  where
    $A_{\varphi_{\eta}}:= \Image(\tilde{\varphi}_{\eta}^{\sharp})
	   \subset C^{\infty}(\End_{\Bbb C}(E))$.
 This shows that $\varphi_{\eta}^{\sharp}$ is $C^{\infty}$-admissible and proves the theorem.

\end{proof}

\bigskip

\section{Remarks on the general $C^k$ case}

Even if not conceptually,
technically the finitely differentiable case seems to be more difficult
   than the smooth (i.e.\ infinitely differentiable) case.
Some remarks are collected here as a guide to verify Conjecture~1.3 and Conjecture~1.5 in full
 (or to correct them, taking the statements as the reference starting point and see how things could break).

\bigskip

\begin{flushleft}
{\bf Reflections on $C^{\infty}$- vs.\ general $C^k$-algebraic geometry, and the proof}
\end{flushleft}
(1)
 {From} the construction of
    the canonical  $C^k$-ring structure on a commutative finite-dimensional ${\Bbb R}$-algebra
  in Sec.~2,
  one learns that
   while $C^{\infty}$-algebraic geometry is self-contained (in the sense that
     only elements in $\cup_{l=0}^{\infty}C^{\infty}({\Bbb R}^l)$ are involved),
   $C^k$-algebraic geometry with $k$ finite may not
    (in the sense that elements in
	    $\cup_{k^{\prime}<k}\cup_{l=0}^{\infty}
		    C^{k^{\prime}}({\Bbb R}^l)$
	    that come from partial derivatives of elements in
		   $\cup_{l=0}^{\infty}C^{\infty}({\Bbb R}^l)$
        are involved as well
		when the $C^k$-scheme considered is not reduced).
			
 \bigskip

(2)
The proof of Conjecture~1.5 in the smooth case (Theorem~3.2.1)
    by first constructing
	     $\varphi_{\eta}^{\sharp}$ and  $\tilde{\varphi}_{\eta}^{\sharp}$
           with values in $C^{-{\infty}}(\End_{\Bbb C}(E))$  and
	     then proving that they actually take values in $C^{\infty}(\End_{\Bbb C}(E))$
 reminds one of wall-crossing phenomena in string theory\footnote{Readers
                           are referred to [C-V] (1993) when such notion started to surface in string theory
						     before  becoming a major study,
                           and to keyword search for more recent fast and vast development
						     from various stringy and/or mathematical aspects,
						     including counting solitonic D-brane systems.}
 in which  some quantities
   (e.g.\ soliton numbers;  here, canonical-form-rendering automorphisms/frames)
 jump cell by cell in order that a related geometric quantity
  (e.g.\ flat sections from solutions to a differential system;
      here, endomorphisms of a complex vector bundle)
 can be kept continuous (here even differentiable).
A simple example serves to illuminate this:

\bigskip

\begin{sexample} {\bf [wall-crossing of frames vs.\ smoothness of endomorphism].} \rm
 Let
  $X^{\!A\!z}$ be the Azumaya/matrix smooth line $({\Bbb R}^{1,A\!z},E)$,
    where $E$ is a complex vector bundle of rank $2$ on $X={\Bbb R}^1$ with coordinate $x$, and     
  $Y={\Bbb R}^1$ be the smooth real line with coordinate $y$.
 For convenience, we assume that $E$ is trivialized.
 Consider the assignment
  $$
    \eta\;:\; y\; \longmapsto\;
	  m\; =\; \left[ \begin{array}{cc} x & \;\;\;1 \\ 0 & -x    \end{array} \right]\;
	                  \in\; C^{\infty}(\End_{\Bbb C}(E))\,.
  $$
 {To} extend $\eta$ to a $C^{\infty}$-admissible ring-homomorphism
  $$
     \varphi_{\eta}^{\sharp}\; ;\;
	   C^{\infty}(Y)\; \longrightarrow\;
	   C^{\infty}(\End_{\Bbb C}(E))\,,
  $$
  consider the following chamber structure on $X$ and
  Jordan-form-of-$m$-rendering frames $(e_1,e_2)$ of $E$ over each chamber:
  $$
  \begin{array}{lcl}
    \mbox{for $x\ne 0$,}
	  &&  (e_1,e_2)\;=\;
	            \left[\begin{array}{cc} 1 & 1 \\ 0 & 2x   \end{array}  \right]\; =:\; S_1\,; \\[3.6ex]
    \mbox{for $x=0$,}
      &&  (e_1,e_2)\;=\;
	            \;\left[\begin{array}{cc} 1 & 0 \\ 0 & 1     \end{array}  \right]\;\; =:\; S_0\; .
  \end{array}
  $$

Let $f\in C^{\infty}(Y)$.
Then,
  for $x\ne 0$, one has
   $$
     m|_{x\ne 0}\;
	     =\;  S_1  \cdot
	                         \left[\begin{array}{cc} x & \;0 \\ 0 & -x\end{array}\right]
							 \cdot S_1^{-1}\,;
   $$
 thus, over each of the two chambers $\{x>0\}$  and $\{x<0\}$,
  $$
    f(m|_{x\ne 0})\;
	  =\;  S_1\cdot
                          f\left(\left[
                                         \begin{array}{cc} x & \;0 \\ 0 & -x \end{array}			
			                 \right]\right)
  	                      \cdot S_1^{-1}\;
      =\;S_1\cdot
                       \left[
                           \begin{array}{cc} f(x) & 0 \\ 0 & f(-x) \end{array}			
			           \right]
  	                 \cdot S_1^{-1}\;       						
      =\; \left[
	         \begin{array}{cc}
			    f(x) &  \frac{f(x)-f(-x)}{2x}\\
			    0 & f(-x)
			 \end{array}
            \right].
  $$
  While at $x=0$,
  $$
     m(0)\;=\; \left[ \begin{array}{cc} 0 & 1 \\ 0 & 0 \end{array}\right]
  $$
  and
  $$
   f(m(0))\;
    =\; f\left( \left[\begin{array}{cc} 0 & 1 \\ 0 & 0\end{array}\right] \right)\;
	=\; \left[\begin{array}{cc} f(0) & f^{\prime}(0) \\ 0 & f(0)    \end{array}\right]\,,
  $$
  where $f^{\prime}(0)=\frac{df}{dy}(0)$.
 Notice that for $f\in C^{\infty}({\Bbb R}^1)$,
  $( f(x)- f(-x)  )/(2x)$ is smooth at $0$
  (an immediate consequence of the Malgrange Division Theorem again) and hence at all $x$,
  with its value at $0$ equal to $f^{\prime}(0)$.
 Thus,
  while the Jordan form $J_m$  of $m$, Jordan-form-of-$m$-rendering $S$,
   and, hence, all the factors in the product $\;S\,f(J_m)\,S ^{-1}\,$
   are discontinuous at $x=0$,
  the product $\;S\,f(J_m)\,S ^{-1}$, which gives $f(m)$,
  remains continuous, even smooth, over all $X$.

\end{sexample}

\bigskip

\begin{flushleft}
{\bf A conjecture on a division lemma in the finitely differentiable case}
\end{flushleft}
In the proof of Theorem~3.1.1 and Theorem~3.2.1,
though differentiability of
   $\tilde{\varphi}^{\sharp}(f)$ or
   $\tilde{\varphi}_{\eta}^{\sharp}(f)\in C^{-\infty}(\End_{\Bbb C}(E))$
  at a point $p\in X$
  is an issue that involves only an infinitesimal neighborhood of $p\in X$,
 technically it looks very difficult to prove it without employing consequences
  from the Malgrange Division Theorem,
  which is a theorem at the level of germs on a small neighborhood of $p\in X$.
(Cf.\ Readers may try to prove directly that
      $(f(x)-f(-x))/(2x)$ for $f\in C^{\infty}({\Bbb R}^1)$
  	  in Example~4.1 is smooth at $x=0$
	  without employing the Malgrange Division Theorem or its similar construction or argument.)
While such theorem looks more than we need,
 if it is indispensable, then one would expect  a version of it in the finitely differentiable case
 would prove both Conjecture~1.3 and Conjecture~1.5
  since all other part of the proof of Theorem~3.1.1 and Theorem~3.2.1 works also for finite $k$.
The following conjecture is guided by the Taylor expansion of a $C^k$-function
  in the normal direction to a codimension-$1$ $C^k$ subscheme:

\bigskip

\begin{sconjecture} {\bf [generalized division lemma in finitely differentiable case].}
 Let
    \begin{itemize}
	  \item[\LARGE $\cdot$]
       $0$ be the origin of $\,{\Bbb R}^{m+1}$
	   (and same notation also for the origin of $\,{\Bbb R}^m$, if necessary),
	
	  \item[\LARGE $\cdot$]
       $y$ be the $(m+1)$-th coordinate of $\,{\Bbb R}^{m+1}={\Bbb R}^m\times {\Bbb R}^1$,
	
	  \item[\LARGE $\cdot$]
       $h\in C^k({\Bbb R}^{m+1}) $ such that
         $$
           h(0)\;=\; \partial_yh(0)\;=\; \cdots\; =\; \partial_y^{\,s-1}h(0)\;=0
		    \hspace{1em}\mbox{while}\hspace{1em}
		   \partial_y^{\,s}h(0)\ne 0
         $$
		 for some $s\le k$.
   \end{itemize}
  Denote by $C^k({\Bbb R}^{m+1})_{(0)}$
      the germs of $C^k$-functions on ${\Bbb R}^{m+1}$ at $0$;
	 and similarly for $C^s({\Bbb R}^m)_{(0)}$.
 Then,
   for all $f\in C^k({\Bbb R}^{m+1})_{(0)}$,
   there exists $g\in C^k({\Bbb R}^{m+1})_{(0)}$
         (or some sensible subset of
		     $\cup_{k^{\prime}=0}^kC^{k^{\prime}}({\Bbb R}^{m+1})_{(0)}$)
     and $a_i\in C^{k-i}({\Bbb R}^m)_{(0)}$, $i=1,\,\ldots\,,\,s$,
   such that
    $$
	    f\; =\;  gh + \sum_{i=1}^s a_iy^{s-i}
	$$
	in $C^k({\Bbb R}^{m+1})_{(0)}$.
\end{sconjecture}

\newpage
\baselineskip 13pt
{\footnotesize

\vspace{1em}

\noindent
chienhao.liu@gmail.com, chienliu@math.harvard.edu; \\
yau@math.harvard.edu

}

\end{document}